\documentclass{article}
\usepackage[utf8]{inputenc}
\usepackage{ mathrsfs }
\usepackage{comment}
\usepackage{tikz}
\usepackage{indentfirst}
\usepackage[hidelinks]{hyperref}
\usepackage{enumitem}

\usepackage{amsrefs}
\usepackage{thm-restate}

\usepackage{amsthm}
\usepackage{amsfonts}
\newcommand{\lp}{\left(}
\newcommand{\rp}{\right)}
\newcommand{\lb}{\left[}
\newcommand{\rb}{\right]}
\newcommand{\Per}{\mathcal P}
\newtheorem{theorem}{Theorem}[section]
\newtheorem{lemma}[theorem]{Lemma}
\newtheorem{definition}[theorem]{Definition}
\newtheorem{corollary}[theorem]{Corollary}

\title{Quasi-isometries of relatively hyperbolic groups with an elementary hierarchy}
\author{Aaron W. Messerla}
\date{}

\begin{document}
\maketitle
\begin{abstract}
    Sela introduced limit groups in his work on the Tarski problem, and showed that each limit group has a cyclic hierarchy. In this paper, a class of relatively hyperbolic groups, equipped with a hierarchy similar to the one for limit groups, is shown to be closed under quasi-isometry. Additionally, these groups share some of the properties of limit groups. In particular, groups quasi-isometric to limit groups are shown to be LERF and virtually toral relatively hyperbolic.
\end{abstract}

%\cite{Wise}{Theorem 7.54}

%\section{Facts About Boundaries}

%\begin{theorem}\label{HrusakEnds}
%[Haulmark 1.3 \cite{HaulmarkLocalCutPoints}] Let $X$ and $Y$ be two generalized Peano continua, and assume that $G$ is a finitely generated group acting properly and cocompactly by homeomorphisms on $X$ and $Y$. Then $\textrm{Ends}\lp X\rp$ is homeomorphic to $\textrm{Ends}\lp Y\rp$.
%\end{theorem}

%In the current setting, this is used to show that a peripheral subgroup $P$ of $\lp G,\Per\rp$ has the same number of ends as the complement of the parabolic point in $\partial\lp G,\Per\rp$ associated to $P$.

%\section{QI invariance of Vertices in Bowditch splitting}

%This is a modification by the arguments of Papasoglu-Whyte to apply in the relatively hyperbolic case.

%%\begin{lemma}
%Let $\lp G,\Per\rp$ be the fundamental group of a graph of groups $\mathcal G$ so that the vertex groups of $\mathcal G$ are relatively one-ended and hyperbolic relative to the elements of $\Per$ that they contain, and the edge groups are finite. Suppose that $\lp H,\Per\rp$ is relatively $1$--ended and $\lp H,\Per'\rp\rightarrow\lp G,\Per\rp$ is a quasi-isometric embedding that coarsely respects peripherals. Then $H$ is contained in a bounded neighborhood of one of the vertex groups of $\mathcal G$.
%\end{lemma}
%\begin{proof}
%Let $X$ be the coned off Cayley graph of $H$ and $Y$ be the coned off Cayley graph of $G$. Then the quasi-isometric embedding $H\rightarrow G$ induces a map $X\rightarrow Y$ that takes infinite valence vertices to within bounded distance of infinite valence vertices. 
%\end{proof}
\section{Introduction}

One of the central questions in geometric group theory is to understand useful quasi-isometry invariants of finitely generated groups. One such invariant is the number of ends of a group. Stallings showed that a group with more than one end splits over a finite subgroup \cite{Stallings}. In the case of hyperbolic groups, the boundary is another useful quasi-isometry invariant.  The topology of the boundary of a hyperbolic group gives information about splittings of the group. Stallings' Theorem on ends of groups indicates that a hyperbolic group splits over a finite group if and only if the boundary is disconnected. In \cite{BowditchCutPoints}, Bowditch showed that cut pairs in the boundary of a hyperbolic group which is not virtually Fuchsian correspond to splittings over $2$--ended subgroups, and constructed a canonical JSJ splitting for such $1$--ended hyperbolic groups. In particular, since the boundary is invariant under quasi-isometry, then so is the property of splitting over a finite or $2$--ended group, except in the special case of groups with boundary homeomorphic to the circle.

Dru{\c t}u showed that the property of being relatively hyperbolic is also a quasi-isometry invariant \cite{DrutuPeriphStab} (see also the work of Behrstock, Dru{\c t}u, and Mosher \cite{BDMPeriphStab}). Associated with a relatively hyperbolic group is a boundary, introduced by Bowditch in \cite{BowditchRH}. Following and generalizing ideas in \cite{BowditchCutPoints}, information about splittings of a relatively hyperbolic group relative to its peripheral structure can be seen from its Bowditch boundary. This was done in the case of splitting over finite groups by Bowditch in \cite{BowditchRH}, over peripheral groups by Bowditch in \cite{BowditchPeriph}, and over $2$--ended groups by Haulmark and Hruska in \cite{HHCanonSplittings}. In particular, the JSJ built in \cite{HHCanonSplittings} is quasi-isometry invariant in the sense that if there is a quasi-isometry between two relatively $1$--ended relatively hyperbolic groups which respects the peripheral structures, then there is an isomorphism between the JSJ trees which preserves information about the types of vertex groups (\cite{HHCanonSplittings}*{Corollary 1.3}).  %It is possible to recover an abelian (or cyclic) hierarchy of a limit group purely from the Bowditch boundary of the group.

Limit groups were introduced by Sela in \cite{Sela1} as part of his work on the Tarski problem. Sela showed that limit groups have cyclic hierarchies (the analysis lattice in Sela's work), which terminate in free, abelian, or surface groups. Using this hierarchy, Dahmani showed in \cite{Dahmani} that limit groups are relatively hyperbolic relative to their maximal non-cyclic abelian subgroups. Alibegovi\'c also proved a combination result for relatively hyperbolic groups, and similarly showed that limit groups are hyperbolic relative to maximal abelian subgroups in \cite{Alibegovic}. Dahmani additionally showed that finitely generated subgroups of limit groups are relatively quasi-convex \cite{Dahmani}*{Theorem 4.6}.

In this paper, a class of groups with hierarchies that resemble those of limit groups are shown to have properties similar to those of limit groups. A group $G$ is said to be in $\mathcal{ATEH}$ if $G$ is hyperbolic relative to virtually abelian subgroups and has a finite elementary hierarchy relative to $\Per$ terminating in virtually free, virtually abelian, or virtually Fuchsian groups (Definition \ref{ATEHDef}). Since the analysis lattice functions as the required elementary hierarchy, any limit group is in $\mathcal{ATEH}$. The first result is that this class of groups is closed under quasi-isometry:

\begin{theorem}\label{QIClosed}
Suppose $G\in \mathcal{ATEH}$ and that $G'$ is quasi-isometric to $G$, then $G'\in\mathcal{ATEH}$.
\end{theorem} 

Restricting Theorem \ref{QIClosed} to the case where $G\in\mathcal{ATEH}$ is hyperbolic (with empty peripheral structure) and has no 2-torsion, one can recover a corollary to work of Carrasco and Mackay \cite{CarrascoMackayConfDim}*{Corollary 1.2}. Specifically, in \cite{CarrascoMackayConfDim} it is shown that if a hyperbolic group $G$ has no 2-torsion, then $G$ has a finite hierarchy over elementary subgroups ending in finite or virtually Fuchsian groups if and only if $\partial G$ has conformal dimension equal to $1$. As the conformal dimension of the boundary is a quasi-isometry invariant, then so is having such a hierarchy.  In this paper, however, we do not use the conformal dimension techniques of \cite{CarrascoMackayConfDim}. 

The structure of the hierarchy guarantees certain algebraic properties of groups in $\mathcal{ATEH}$. In particular, like limit groups, if $G\in \mathcal{ATEH}$, then $G$ is locally relatively quasi-convex (i.e. every finitely generated subgroup of $G$ is relatively quasi-convex). Special cube complexes, introduced by Haglund and Wise in \cite{HaglundWise}, have been studied extensively as a means to recover separability properties in virtually compact special groups. Wise showed that limit groups are virtually compact special \cite{WiseQCH}*{Theorem 18.7}. In \cite{WiltonHallsTheoremLimitGroups}, Wilton showed that limit groups are LERF. Results from the theory of special cube complexes from \cite{WiseQCH},\cite{oregonvspecial}, and \cite{SageevWiseCores}  are used in this paper to show the same for groups in $\mathcal{ATEH}$.

\begin{theorem}\label{ATEHisLerf}
Suppose that $G\in\mathcal{ATEH}$, then $G$ is LERF.
\end{theorem} 

As limit groups are torsion-free and hyperbolic relative to abelian groups, it is of interest to investigate a smaller class of groups in $\mathcal{ATEH}$, namely those that are toral relatively hyperbolic (i.e. torsion-free and hyperbolic relative to abelian subgroups). We call the class of toral relatively hyperbolic groups with an elementary hierarchy and controlled terminal groups $\mathcal{TEH}$. Groups in $\mathcal{ATEH}$ are shown to be virtually torsion-free, and the separability properties of groups in $\mathcal{ATEH}$ allow us to show:

\begin{restatable}{theorem}{ThmATEHisVTEH}\label{ATEHisVTEH}
Suppose that $G\in\mathcal{ATEH}$, then $G$ is virtually in $\mathcal{TEH}$.
\end{restatable}

The following is immediate from Theorems 1.1 and 1.3, and the fact that limit groups are in $\mathcal{ATEH}$. 

\begin{corollary}
If $G$ is quasi-isometric to a limit group, then $G$ is virtually in $\mathcal{TEH}$.
\end{corollary}

Lastly, freely indecomposable and non-abelian limit groups were shown to split over cyclic subgroups by Sela \cite{Sela1}*{Theorem 3.2}. A comparable result is shown for groups in $\mathcal{ATEH}$. This allows us to show that, like limit groups, groups in $\mathcal{ATEH}$ actually have a finite hierarchy where all splittings are over virtually cyclic groups.

\begin{theorem}
 Suppose that $G\in\mathcal{ATEH}$, then $G$ has a finite virtually cyclic hierarchy terminating in virtually free, virtually abelian, or virtually Fuchsian groups.
\end{theorem}

This paper is organized as follows. Section \ref{SectionPrelims} consists of background on relatively hyperbolic groups, distortion, and LERFness and residual finiteness. Section \ref{SectionSplittings} reviews splittings, JSJ decompositions, and previously known results relating the boundary of a relatively hyperbolic group and splittings of that group. Section \ref{SectionQIResults} summarizes known quasi-isometry results needed, and proves Lemma \ref{BulkLemma}, which is used for inductive proofs through the remainder of the paper. Section \ref{SectionHierarchies} recalls the definition of a hierarchy, and introduces the class $\mathcal{ATEH}$. In Section \ref{SectionBHHJSJ}, the BHH--JSJ is defined, and used to show that the class $\mathcal{ATEH}$ is closed under quasi-isometry. In Section \ref{SectionSeparability}, groups in $\mathcal{ATEH}$ are shown to be locally relatively quasi-convex, LERF, virtually torsion-free, and virtually in $\mathcal{TEH}$. Lastly, in Section \ref{SectionCyclicHierarchies}, it is shown that groups in $\mathcal{ATEH}$ not only have virtually abelian hierarchies, but also have finite virtually cyclic hierarchies as well.

\subsection*{Acknowledgments} 
\vspace{-.75pt} The author would like to thank his advisor, Daniel Groves, for suggesting this project and guidance toward its completion, as well as Benson Farb for helpful comments about references.

\section{Preliminaries}\label{SectionPrelims}
We refer the reader to \cite{BHMetricSpaces} for a background and definitions of $\delta$-hyperbolic metric spaces and quasi-isometries.

We briefly recall the definition of a relatively hyperbolic group given in \cite{GrovesManningFillings}. If $\Gamma$ is a graph, we will always assume that $V\lp \Gamma\rp$ is the set of vertices and $E\lp \Gamma\rp$ is the set of edges. Additionally, we assume each edge of $\Gamma$ has length $1$, and give $\Gamma$ the induced path metric, which we denote $d_\Gamma$.

\begin{definition}[Combinatorial Horoball]
Let $\Gamma$ be a connected graph. The \emph{combinatorial horoball} over $\Gamma$ is a graph $\hat{\Gamma}$ with $V\lp \hat{\Gamma}\rp=V\lp \Gamma\rp\times \mathbb Z_{\geq 0}$. There are two types of edges, given by \begin{itemize}
    \item an edge between $\lp g,i\rp$ and $\lp g,i+1\rp$ for all $g\in V\lp \Gamma\rp$ and $i\in \mathbb Z_{\geq 0}$, and
    \item   an edge between $\lp g,i\rp$ and $\lp h,i\rp$ if $0<d_{\Gamma}\lp g,h\rp\leq 2^i$.
\end{itemize}
\end{definition}  

We note that $\hat{\Gamma}$ contains an isomorphic copy of $\Gamma$ as the full subgraph on $V\lp\Gamma\rp\times\{0\}$.

Suppose $G$ is a finitely generated group, $\Per=\{P_1,...,P_n\}$ a finite set of infinite finitely generated subgroups. We say that $\mathcal S$ is a \emph{compatible generating set} if $\mathcal S$ is a generating set of $G$ so that $\mathcal S\cap P_i$ generates $P_i$ for each $i$.

\begin{definition}[Cusped Cayley graph]
Let $G$ and $\Per$ be as above. The \emph{cusped Cayley graph} $X\lp G,\Per,\mathcal S\rp$ is a quotient of the disjoint union of the Cayley graph of $G$ with respect to a compatible generating set $\mathcal S$ and for each $i$ a combinatorial horoball over $P_i$ for every coset of $P_i$. The horoballs are attached to the Cayley graph so that every vertex $gp\in gP_i$ in the Cayley graph is identified with $\lp p,0\rp$ in the combinatorial horoball over $P_i$ associated with $gP_i$. 
\end{definition} 
For details of this construction see \cite{GrovesManningFillings}*{Section 3}.

\begin{definition}
With $G$ and $\Per$ as above, $\lp G,\Per\rp$ is \emph{relatively hyperbolic} if $X\lp G,\Per,\mathcal S\rp$ is $\delta$-hyperbolic for some $\delta$ and some compatible generating set $\mathcal S$.
\end{definition}
 
  It is shown in \cite{GrovesManningFillings}*{Theorem 3.25} that $\lp G,\Per\rp$ is relatively hyperbolic in the sense of Gromov \cite{GromovHG} if and only if $X\lp G,\Per,\mathcal S\rp$ is hyperbolic for some compatible generating set $\mathcal S$,  establishing that the definition above does not depend on the choice of $\mathcal S$. We note that by \cite{GroffBoundaries}*{Theorem 6.3}, the cusped Cayley graphs with two different compatible generating sets are quasi-isometric.
 
 The boundary of a hyperbolic space is discussed in detail in \cite{BHMetricSpaces}*{Section III.H.3}, but we recall the definition here. Recall a \emph{quasi-geodesic ray} in a metric space $X$ is a quasi-isometric embedding of $\lb 0,\infty\rp$ into $X$.
 
 \begin{definition}
 Let $X$ be a hyperbolic metric space. The boundary of $X$, denoted $\partial X$, is defined to be the set of equivalence classes of infinite quasi-geodesic rays, where two such rays are in the same equivalence class if the Hausdorff distance between their images is finite.
 \end{definition}

If $\lp G,\Per\rp$ is relatively hyperbolic, the space $X\lp G,\Per,\mathcal S\rp$ is not dependant on the  choice of $\mathcal S$ up to quasi-isometry. Therefore, if $\mathcal S$ and $\mathcal S'$ are two compatible generating sets, $\partial X\lp G,\Per,\mathcal S\rp$ and $\partial X\lp G,\Per,\mathcal S'\rp$ are homeomorphic by \cite{BHMetricSpaces}*{Theorem III.H.3.9}. Additionally, \cite{YamanTop}*{Theorem 0.1} guarantees the boundary of $X\lp G,\Per,\mathcal S\rp$ is equivariantly homeomorphic to the boundary $\partial\lp G,\Per\rp$ as introduced in \cite{BowditchRH}.

If $\lp G,\Per\rp$ is relatively hyperbolic with generating set $S$, the action of $G$ on $G$ by left multiplication takes cosets of $P\in \Per$ to cosets of $P$, so the action extends to the space $X\lp G,\Per, S\rp$, and therefore to the boundary $\partial\lp G,\Per\rp$. A subgroup $H$ of $G$ is said to be \emph{elementary} if $H$ fixes a point $p\in \partial\lp G,\Per\rp$. If $H$ is elementary either $H$ is $2$--ended, or $H$ is conjugate into some $P\in \Per$. 

Most other general facts about relatively hyperbolic groups used in this paper can be found in \cite{OsinRH}.

Next, we recall the notion of distortion. For more details see \cite{GGTDrutuKapovich}*{Section 8.9} or \cite{HruskaDistortion}. For functions $f,g:\lb 0,\infty\rp\rightarrow \lb 0,\infty\rp$, say $f\preceq g$ if there exists a $C$ so that for all $r\geq 0,$ $f\lp r\rp\leq Cg\lp Cr+C\rp+Cr+C$ . Write $f\cong g$ if $f\preceq g$ and $g\preceq f$. Note that $\cong$ is an equivalence relation. If $f\lp n\rp=n$ and $g\lp n\rp=0$ then $f\cong g$. 

\begin{definition}[Distortion]
Suppose $G$ is a finitely generated group with generating set $\mathcal S$ and $H\leq G$ is finitely generated with generating set $\mathcal T$. The \emph{distortion} of $\lp H,\mathcal T\rp$ in $\lp G,\mathcal S\rp$ is  $$\Delta_H^G\lp n\rp:=\max\{|h|_\mathcal{T}: h\in H \textrm{ and } |h|_\mathcal{S}\leq n\}.$$
\end{definition} 

\begin{definition}[Undistorted]
A subgroup is \emph{undistorted} if and only if $\Delta_H^G\cong\lp n\mapsto n\rp$.
\end{definition}

As noted in \cite{GGTDrutuKapovich}*{Section 8.9} a subgroup is undistorted if and only if the inclusion map is a quasi-isometric embedding. The following lemma is easy from this fact and \cite{GGTDrutuKapovich}*{Proposition 8.98 (6)}.

\begin{lemma}
All subgroups of a finitely generated virtually abelian group are undistorted.
\end{lemma}

The following lemma involves the notion of relatively quasi-convex subgroups of a relatively hyperbolic group $\lp G,\Per\rp$. There are multiple equivalent definitions of relatively quasi-convex subgroups. We refer the reader to \cite{HruskaDistortion} for the various definitions as well as the proofs of their equivalence. For this paper, we take Definition (QC-1) from \cite{HruskaDistortion}. A subgroup $H\subseteq G$ is \emph{relatively quasi-convex} if the induced action of $H$ on the limit set of $H$, $\Lambda H$, is geometrically finite. Note that if $H$ is relatively quasi-convex, then by the work of Yaman \cite{YamanTop}, $H$ is relatively hyperbolic, and the limit set $\Lambda H$ is equivariantly homeomorphic to $\partial \lp H, \mathcal{O}\rp$ where $\mathcal O$ is the induced peripheral structure. 
 Immediately from the preceding lemma and from \cite{HruskaDistortion}*{Theorem 10.5}, we get the following lemma.

\begin{lemma}\label{HruskaDistorted}
Suppose that $\lp G,\Per\rp$ is relatively hyperbolic and every $P\in \Per$ is virtually abelian. Then every relatively quasi-convex subgroup of $\lp G,\Per\rp$ is quasi-isometrically embedded in $G$.
\end{lemma}

The following lemma is well known and immediate from work of Osin \cite{OsinRH}. In this paper it is used to remove $2$--ended groups from the peripheral structure induced on vertex groups in an elementary splitting.

\begin{lemma}\label{ThrowOut2ended}
Suppose that $\lp G,\Per\rp$ is relatively hyperbolic, and that $\{P_i\}_{i=1}^n\subseteq \Per$ are hyperbolic. If $\Per'=\Per\setminus\{P_i\}_{i=1}^n$, then $\lp G,\Per'\rp$ is relatively hyperbolic.
\end{lemma}
\begin{proof}
This follows from \cite{OsinRH}*{Theorem 2.40}, as relatively hyperbolic groups have linear relative Dehn functions and hyperbolic groups have linear (ordinary) Dehn functions.
\end{proof}

In Section \ref{SectionSeparability}, we show that elements of the class of groups being studied are LERF. A subgroup $H$ of $G$ is said to be \emph{separable} if it is closed in the profinite topology. This is equivalent to $H$ being the intersection of finite index subgroups of $G$, or having the property that for any $g\in G\setminus H$, there is a finite group $Q$ and a homomorphism $\phi:G\rightarrow Q$ so that $\phi\lp g\rp\not\in\phi\lp H\rp$. A group $G$ is \emph{LERF} (or \emph{subgroup separable}) if any finitely generated subgroup of $G$ is separable. Classes of groups which are LERF include free groups (as a consequence of Hall's Theorem \cite{HallsTheorem}), surface groups \cite{ScottsSurfaceLerf}, and limit groups \cite{WiltonHallsTheoremLimitGroups}. We refer the reader to \cite{WiltonHallsTheoremLimitGroups} for more background and additional references. 

A group $G$ is said to be \emph{residually finite} if for each $1\neq g\in G$, there is a finite index normal subgroup $N$ of $G$ such that $g\not\in N$. We note that if $G$ is LERF, then $G$ is residually finite, as $\{1\}$ is a finitely generated subgroup of $G$, so it is the intersection of finite index subgroups of $G$. In particular, there is some finite index normal subgroup missing any non-identity element of $G$.

\section{Splittings}\label{SectionSplittings}

By a \emph{splitting} of a group, we mean a realization of $G$ as the fundamental group of a finite graph of groups. A splitting is said to be \emph{over} a class $\mathcal{E}$ if every edge group in the splitting is an element of $\mathcal{E}$. A splitting of a group $G$ is said to be \emph{relative to $\Per$} if every $P\in\Per$ fixes a vertex of the Bass-Serre tree. A splitting of a group $G$ is said to be \emph{trivial} if $G$ is one of the vertex groups. 

Next we recall the notion of a JSJ decomposition. The general definition which will be used in this paper appears in \cite{GLJSJs}. An $\mathcal{E}$-tree is the Bass-Serre tree corresponding to a splitting of a group $G$ over the class of subgroups $\mathcal E$. An $\lp \mathcal{E},\Per\rp$-tree is an $\mathcal{E}$-tree so that all groups in $\Per$ act elliptically on the tree. 

\begin{definition}[Universally Elliptic]
An $\lp \mathcal{E},\Per\rp$-tree $T$ is \emph{universally elliptic} if every edge stabilizer $G_e$ acts elliptically on any other $\lp \mathcal{E},\Per\rp$ tree.
\end{definition}

Given two trees $T$ and $T'$ that $G$ acts on, we say that $T$ \emph{dominates} $T'$ if there is a $G$-equivariant map $f:T\rightarrow T'$.

\begin{definition}[JSJ tree]
An $\lp \mathcal{E},\Per\rp$-tree $T$ is a \emph{JSJ-tree} if $T$ is universally elliptic and dominates any other universally elliptic $\lp \mathcal{E},\Per\rp$-tree.
\end{definition}

The next two theorems are due to Bowditch \cite{BowditchRH}. The first relates splittings over finite groups relative to the peripheral structure to connectedness of the boundary, the second is an accessibility result. 

\begin{theorem}[\label{BowditchConnectedness}{\cite{BowditchRH}*{Theorem 10.1}}] The boundary $\partial\lp \Gamma,\Per\rp$ of a relatively hyperbolic group $\lp \Gamma,\Per\rp$ is connected if and only if $\Gamma$ does not non-trivially split over finite subgroups relative to $\Per$.
\end{theorem}

\begin{theorem}
[\label{BowditchAccess}\cite{BowditchRH}*{Theorem 10.2}] Any relatively hyperbolic group $\lp \Gamma,\Per\rp$ can be expressed as the fundamental group of a finite graph of groups with finite edge groups, where the splitting is relative to $\Per$, with the property that no vertex group splits non-trivially over any finite subgroup relative to the peripheral subgroups.
\end{theorem}

In light of Theorem \ref{BowditchConnectedness}, we say that a relatively hyperbolic group $\lp G,\Per\rp$ is \emph{relatively $1$--ended} if $G$ does not split over a finite group relative to $\Per$.

In order to apply the work of Haulmark and Hruska \cite{HHCanonSplittings}, we need the boundary of a relatively hyperbolic group to be locally connected. The next theorem, due to Dasgupta and Hruska, is a generalization of a result of Bowditch \cite{BowditchPeriph}*{Theorem 1.5}. The result of Bowditch would have been sufficient in our case to apply the theorems of Haulmark and Hruska that are summarized below as Theorem \ref{HHJSJ}, but the more general version is stated here.  

\begin{theorem}[\label{ConnectedGivesLocallyConnected}{\cite{DasgHruskLocallyConnected}*{Theorem 1.1}}] If $\lp G,\Per\rp$ is a relatively hyperbolic group with connected boundary $M=\partial\lp G,\Per\rp$, then $M$ is locally connected.
\end{theorem}

As observed in \cite{BowditchPeriph}, Theorems \ref{BowditchConnectedness} and \ref{BowditchAccess} combined results in the fact that each vertex group in the splitting from \ref{BowditchAccess} is relatively $1$--ended.

\begin{theorem}[\label{HHJSJ}\cite{HHCanonSplittings}*{Theorem 1.1, Theorem 8.5, and Proposition 8.7}] Let $\lp G,\Per\rp$ be a relatively hyperbolic group. Suppose that the Bowditch boundary $M=\partial\lp G,\Per\rp$ is connected and locally connected. The canonical JSJ tree of cylinders for splittings of $G$ over elementary subgroups relative to peripheral subgroups is a tree $T\lp M\rp$ that depends only on the topological structure of the boundary M.
Moreover, the tree $T\lp M\rp$ has vertex set $Z\sqcup \Pi$ where $Z$ is the set of all global cut points and inseparable exact cut pairs, and $\Pi$ is the collection of pieces of $M$.
Furthermore, every vertex stabilizer of $T\lp M\rp$ is the one of the following types:
\begin{enumerate}
    \item peripheral
    \item non-parabolic $2$--ended
    \item quadratically hanging with finite fiber, or
    \item rigid.
\end{enumerate}
\end{theorem}

In particular, if $H$ is a vertex group of the third type, it is shown in \cite{HHCanonSplittings} that $H$ fits into a short exact sequence $$1\rightarrow F\rightarrow H\rightarrow \pi_1\lp \Sigma\rp\rightarrow 1,$$ where $F$ is finite and $\Sigma$ is a hyperbolic two-orbifold. We note that the proof of the existence of this short exact sequence relies on the Convergence Group Theorem \cites{Tukia,Gabai,CassonJungreis}.

In this paper, we are considering hierarchies where virtually free and virtually Fuchsian groups are terminal (see Section \ref{SectionHierarchies} and particularly Definition \ref{ATEHDef}). The following lemma will be used to see that vertex groups of the third type in Theorem \ref{HHJSJ} are terminal. This lemma is surely known to the experts, but a proof is included for completeness.

\begin{lemma}
    Suppose that $H$ is a vertex group of the third type in Theorem \ref{HHJSJ}, then either $H$ is virtually free or virtually Fuchsian.
\end{lemma}
\begin{proof}
    Let $$1\rightarrow F\rightarrow H\rightarrow \pi_1\lp \Sigma\rp\rightarrow 1,$$ be the short exact sequence as above. Let $\Sigma'$ be a finite cover of $\Sigma$ so that $\Sigma'$ is a hyperbolic surface. If $\Sigma'$ is not closed, then $\pi_1\lp \Sigma'\rp$ is virtually free. As $H$ is quasi-isometric to a free group, it has boundary homeomorphic to a Cantor set. As explained in \cite{KapovichBenakli}*{Theorem 8.1}, Dunwoody's accessibility result \cite{DunwoodyAccess} can be used to show that $H$ acts on a tree with finite vertex stabilizers, and so \cite{treesbyserre}*{II.2.6 Proposition 11} shows $H$ is virtually free.
    
    In the case that $\Sigma'$ is closed, $\pi_1\lp \Sigma'\rp$ is LERF by \cite{ScottsSurfaceLerf}. As $\pi_1\lp \Sigma'\rp$ is finite index in $\pi_1\lp \Sigma\rp$, $\pi_1\lp \Sigma\rp$ is also LERF, and therefore residually finite (we note that surface groups are residually finite see, for example, the work of Hempel in \cite{Hempel}). Additionally, $\pi_1\lp \Sigma\rp$ is good in the sense of Serre by \cite{GJZZGoodness}*{Proposition 3.7} and \cite{SerreCohomology}*{Chapter 1, Section 2.6}. Therefore, since $H$ fits in a short exact sequence $$1\rightarrow F\rightarrow H\rightarrow \pi_1\lp \Sigma\rp\rightarrow 1,$$  $H$ is residually finite by \cite{GJZZGoodness}*{Proposition 6.1}. Therefore, there is an $N$ which is finite index and normal in $H$ such that $N\cap F=\{1\}$. Then $N$ embeds in a Fuchsian group, so $H$ is virtually Fuchsian.
\end{proof}

\section{Quasi-Isometry Results}\label{SectionQIResults}
The goal of this section is to prove Lemma \ref{BulkLemma}, which is used to show that the class of groups being studied is closed under quasi-isometry. The first result is a combination of known quasi-isometry results.

\begin{lemma}\label{RelVAb}
Suppose that $\lp G,\Per\rp$ is relatively hyperbolic and that each $P\in \Per$ is virtually abelian and $1$--ended. Let $G'$ be a group quasi-isometric to $G$. Then there is a collection of subgroups $\Per'$ so that $\lp G',\Per'\rp$ is relatively hyperbolic and each $P'\in\Per'$ is virtually abelian and $1$--ended.
\end{lemma}
\begin{proof}
The existence of a peripheral structure on $G'$ is given by \cite{BDMPeriphStab}*{Theorem 4.8}. Furthermore, since each $P'\in \Per'$ is finitely generated and quasi-isometric to some $P\in\Per$, $P'$ is virtually abelian by \cite{BridsonGertsen}*{Theorem 5.9}. Since ends are preserved by quasi-isometry, each group in $\Per'$ is also $1$--ended.
\end{proof}

%\begin{theorem}\label{QI1endedVertexTypes}[ {\cite{PapasogluWhyteQIvertexgroups}*{Theorem 0.4}]
%Let $G$ be an accessible group, and let $\pi_1\lp \Gamma,\mathcal A\rp$ be a terminal graph of groups decomposition of $G$. A group $G'$ is quasi-isometric to $G$ if and only if it is also accessible and any terminal decomposition of $G'$, $\pi_1\lp \Delta,\mathcal B\rp$ has the same set of quasi-isometry types of one ended factors and the same number of ends.
%\end{theorem}

We need the following results of Mackay and Sisto \cite{MackaySistoMaps} to be able to compare vertex groups in a JSJ splitting of two quasi-isometric groups. The second result is implicit in \cite{MackaySistoMaps}, but it was not explicitly stated. The only difference between Lemma \ref{MSInToOut} below and \cite{MackaySistoMaps}*{Theorem 1.2} is the observation that the proof of \cite{MackaySistoMaps}*{Theorem 1.2} shows that the boundary map is also shadow respecting.

\begin{lemma}[\label{MSOutToIn}\cite{MackaySistoMaps}*{Theorem 1.3}]
Suppose that $\lp G,\Per\rp$ and $\lp G',\Per'\rp$ are relatively hyperbolic groups with infinite proper peripheral groups, and that $h:\partial_\infty\lp G,\Per\rp\rightarrow\partial_\infty \lp G',\Per'\rp$ is a shadow respecting quasisymmetry. Then there exists a quasi-isometry $\hat{h}:G\rightarrow G'$ so that $\hat{h}$ extends to give $\partial_\infty\lp \hat{h}\rp=h$.
\end{lemma}

\begin{lemma}\label{MSInToOut}
Suppose that $\lp H,\mathcal{O}\rp$ and $\lp G,\Per\rp$ are relatively hyperbolic groups, and that $f:H\rightarrow G$ is a quasi-isometric embedding that coarsely respects the peripheral structures. Then $f$ induces a shadow respecting quasisymmetric embedding $\partial_{\infty} f:\partial\lp H,\mathcal{O}\rp\rightarrow \partial\lp G,\Per\rp.$
\end{lemma}
\begin{proof}
As $f$ is a quasi-isometric embedding which coarsely respects the peripheral structure, $f$ snowflakes on peripherals (\cite{MackaySistoMaps}*{Def. 2.3}), so we can apply \cite{MackaySistoMaps}*{Theorem 2.8}.
\end{proof}

\begin{lemma}\label{BulkLemma}
Suppose that $\lp G,\Per\rp$ is relatively hyperbolic and each $P\in\Per$ is virtually abelian and $1$--ended. Let $G'$ be any group quasi-isometric to $G$, and let $\Per'$ be as in Lemma \ref{RelVAb}.
\begin{enumerate}[label=(\arabic*)]
    \item\label{Rel1end} The boundaries $\partial\lp G,\Per\rp$ and $\partial\lp G',\Per'\rp$ are homeomorphic. In particular, $\lp G',\Per'\rp$ is relatively $1$--ended if and only if $\lp G,\Per\rp$ is.
    \item\label{QI1endedVertexGroups} If $\partial\lp G,\Per\rp$ has infinitely many ends, then each vertex in the maximal splitting of Bowditch is quasi-isometric to a vertex group in the maximal splitting of $\lp G',\Per'\rp$.
    \item \label{QIRigidVertexGroups} Suppose that $\lp G,\Per\rp$ is relatively $1$--ended. If $H\subset G$ is a rigid vertex group in the JSJ decomposition given by Theorem \ref{HHJSJ}, then there is a rigid vertex group $H'$ in the JSJ decomposition of $G'$ which is quasi-isometric to $H$.
\end{enumerate}
\end{lemma}
\begin{proof}
 Let $f:G\rightarrow G'$ and $g:G'\rightarrow G$ be quasi-isometries that are quasi-inverses. Note that $f$ and $g$ coarsely respect peripherals. By \cite{MackaySistoMaps}*{Theorem 2.9}, there are shadow respecting quasisymmetries $\hat{f}:\partial\lp G,\Per\rp\rightarrow \partial\lp G',\Per'\rp$ and $\hat{g}:\partial\lp G',\Per'\rp\rightarrow \lp G,\Per\rp$. The composition $\hat{g}\circ\hat{f}$ is equal to the induced map $\widehat{g\circ f}:\partial\lp G,\Per\rp\rightarrow \lp G,\Per\rp$. As $g\circ f$ is within bounded distance of the identity on the cusped space of $\lp G,\Per\rp$, $\hat{g}\circ\hat{f}$ is the identity on $\partial\lp G,\Per\rp$. Similarly $\hat{f}\circ\hat{g}$ is the identity on $\partial\lp G',\Per'\rp$. Therefore $\hat{f}$ and $\hat{g}$ are homeomorphisms, so $\partial \lp G',\Per'\rp$ is connected if and only if $\partial\lp G,\Per\rp$ is. This shows \ref{Rel1end}.

To show \ref{QI1endedVertexGroups}, recall that Dunwoody in \cite{DunwoodyAccess} showed that all finitely presented groups are accessible, i.e. that there is a splitting of $G$ so that each vertex group is $1$--ended. Since each $P\in \Per$ is $1$--ended, $P$ cannot split over a finite group, and therefore $P$ acts elliptically on the maximal splitting of Dunwoody. Therefore, the maximal splittings given by Dunwoody and Bowditch coincide, and the same can be said about $G'$. Now, \cite{PapasogluWhyteQIvertexgroups}*{Theorem 0.4} shows that the Bowditch splittings of $\lp G,\Per\rp$ and $\lp G',\Per'\rp$ have the same set of quasi-isometry types of vertex groups.

Finally we show \ref{QIRigidVertexGroups}. All relatively quasi-convex subgroups of both $\lp G,\Per\rp$ and $\lp G',\Per'\rp$ are quasi-isometrically embedded by Lemma \ref{HruskaDistorted}.

Now, by Theorem \ref{MSInToOut} there is a shadow respecting quasisymmetry between $\partial\lp G,\Per\rp$ and $\partial\lp G',\Per'\rp$. Note that this is a homeomorphism and therefore must preserve the set of inseparable exact cut pairs.

Let $H$ be a rigid vertex group of the JSJ of $\lp G,\Per\rp$. As in the proof of \cite{HHCanonSplittings}*{Theorem 8.5}, take $\hat{\Per}$ to be the union of $\Per$ with the $G$-stabilizers of exact inseparable cut pairs in $\partial\lp G,\Per\rp$. Thus $\partial\lp G,\hat{\Per}\rp$ is connected, and the collection of pieces of $\partial\lp G,\Per\rp$ are in bijection with the collection of non-trivial cyclic elements of $\partial\lp G,\hat{\Per}\rp$. Now, the rigid vertex groups of the JSJ are exactly the non-quadratically hanging stabilizers of cyclic elements of $\partial\lp G,\hat{\Per}\rp$ (\cite{HHCanonSplittings}*{Lemma 8.6}). Let $\mathcal{O}$ be the set of infinite subgroups of the form $H\cap gPg^{-1}$ for some $g\in G$ and $P\in\hat{\Per}$. We note that $\lp H,\mathcal{O}\rp$ is relatively hyperbolic by \cite{HruskaDistortion}*{Theorem 9.1}. Now, by \cite{HHCanonSplittings}*{Proposition 8.2(2)}, $\partial\lp H,\mathcal O\rp$ is $G$-equivariantly homeomorphic to $C$ where $C$ is the non-trivial cyclic element of $\partial\lp G,\hat{\Per}\rp$ that $H$ stabilizes.

Since the set of inseparable exact cut pairs is preserved under homeomorphism, we can construct $\hat{\Per'}$ by adding $G'$ stabilizers of exact inseparable cut pairs. Since $\partial\lp G',\hat{\Per'}\rp$ is the quotient of $\partial\lp G,\Per'\rp$ identifying the inseparable exact cut pairs, there is a homeomorphism between $\partial\lp G,\hat{\Per}\rp$ and $\partial\lp G',\hat{\Per'}\rp$. This homeomorphism preserves non-trivial cyclic elements. Additionally, by observing that the cyclic component corresponding to $H$ is not a necklace, there is a corresponding rigid vertex group $H'$ in the JSJ tree for $\lp G',\Per'\rp$.

Since each relatively quasi-convex subgroup $H$ of $G$ is quasi-isometrically embedded, Theorem \ref{MSInToOut} gives a shadow respecting quasi-symmetric embedding which asymptotically snowflakes from $\partial\lp H,\mathcal O\rp$ into $\partial\lp G,\Per\rp$ and $\partial\lp H',\mathcal O'\rp$ into $\lp G',\Per'\rp$. The image of both of these embeddings is exactly the cyclic element $C$. Since the property of being a shadow respecting quasi-symmetric embedding is preserved under restriction of codomain, this gives a shadow respecting quasi-symmetry $f$ between $\partial\lp H,\mathcal O\rp$ and $C$, and $g$ between $\partial\lp H',\mathcal O'\rp$ and $C$, where $C$ is given the induced metric structure as a subset of $\partial\lp G,\Per\rp$, which is quasisymmetric to the metric structure as a subset of $\partial\lp G',\Per'\rp$. Therefore, $g^{-1}\circ f$ is a shadow respecting quasi-symmetry from $\partial\lp H,\mathcal O\rp$ to $\partial\lp H',\mathcal O'\rp$. By Theorem \ref{MSOutToIn}, we see that $H$ and $H'$ are quasi-isometric as claimed.
\end{proof}

\section{Hierarchies}\label{SectionHierarchies}
A \emph{hierarchy} $\mathcal H$ of a group $G$ is a rooted tree, with each edge oriented away from the root, and each vertex is labeled by a pair $\lp H,\Delta_H\rp$, where $H$ is a subgroup of $G$ and $\Delta_H$ is a splitting of $H$ satisfying the following conditions:
\begin{itemize}
    \item The root of $\mathcal H$ is labeled by $\lp G,\Delta_G\rp$ for some splitting $\Delta_G$.
    \item If $\lp H,\Delta_H\rp$ labels a vertex and $\Delta_H$ is non-trivial, there is a bijective correspondence between vertex groups of $\Delta_H$ and edges from $\lp H,\Delta_H\rp$ to $\lp H',\Delta_{H'}\rp$ where $H'$ is a vertex group of $\Delta_H$.
    \item If $\lp H,\Delta_H\rp$ labels a vertex of $\mathcal H$ and $\Delta_H$ is a trivial splitting, then $\lp H,\Delta\rp$ has no descendants, in which case we say $\lp H,\Delta_H\rp$ is \emph{terminal}
\end{itemize}

We note that the definition of a hierarchy could have had each splitting $\Delta$ be replaced by the Bass-Serre tree of the splitting as was done in \cite{HillLTforHyperbolic}. This approach would have required an additional ``conjugacy invariance" assumption in exchange for getting an action of $G$ on the tree underlying $\mathcal H$. However, as we don't use the action of $G$ on a hierarchy in this paper, we have chosen the above definition.

The \emph{depth} of a vertex $\lp H, \Delta_H\rp$ in $\mathcal H$ is the distance from the root of $\mathcal H$ to $\lp H,\Delta_H\rp$. Let $\mathcal H^i$ be the set of vertices of depth $i$. We say that a hierarchy $\mathcal H$ is \emph{finite} if $\mathcal H^i$ is empty for sufficiently large $i$. The \emph{length} of a hierarchy $\mathcal H$ is the smallest $i$ such that $\mathcal H^j$ is empty for all $j\geq i$. 

 Let $\mathcal {E}$ be a class of subgroups closed under taking subgroups. A hierarchy $\mathcal H$ is said to be \emph{over} $\mathcal E$ if every splitting is over $\mathcal E$. Suppose $\lp G,\Per\rp$ is relatively hyperbolic. If $\mathcal H$ is a hierarchy over the class of elementary subgroups, then each vertex group of $\Delta_G$ is relatively quasi-convex by \cite{HHCanonSplittings}*{Theorem 5.2}. Let $H$ be a vertex group of $\Delta_G$ and let $\mathcal O'_H$ denote the set of infinite subgroups of the form $gPg^{-1}\cap H$. Then $\lp H,\mathcal O'_H\rp$ is relatively hyperbolic. Let $\mathcal O_H$ be $\mathcal O'_H$ with any $2$--ended elements removed. Then $\lp H,\mathcal O_H\rp$ is relatively hyperbolic by Lemma \ref{ThrowOut2ended}. Continuing down the hierarchy, we get a peripheral structure on every group labeling a vertex in a hierarchy over elementary subgroups. Say that a hierarchy $\mathcal H$ over elementary subgroups is \emph{relative} to $\Per$ if for each vertex group $\lp H,\Delta_H\rp$ and $\mathcal O_H$ as above, each $O\in \mathcal O_H$ acts elliptically on the tree associated to $\Delta_H$. 

%Suppose that $\lp H,\Delta_H$ is a vertex in the hierarchy $\mathcal H$. Let $\lp H,\Delta_H\rp$ be a vertex in the hierarchy $\mathcal H$. If $gPg^{-1}\cap H$ acts elliptically on the tree associated to $\Delta_H$ whenever $gPg^{-1}\cap H$ is not virtually cyclic or finite, then the hierarchy is said to be \emph{relative} to $\Per$.

\begin{definition}[Almost Toral with Elementary Hierarchy]\label{ATEHDef}
Say that $G\in\mathcal{ATEH}$ if the following holds:
\begin{enumerate}
    \item $G$ is hyperbolic relative to $\Per$ where each $P\in \Per$ is virtually abelian
    \item $G$ has a finite hierarchy $\mathcal H$ over elementary subgroups and $\mathcal H$ is relative to $\Per$
    \item Each terminal vertex of $\mathcal H$ is labeled by $\lp H,\Delta_H\rp$ where $H$ is either virtually free, virtually abelian, or virtually Fuchsian. 
\end{enumerate}
\end{definition}

Clearly the hierarchy required is over virtually abelian subgroups.

\begin{definition}[Toral with Elementary Hierarchy]
Suppose $\lp G,\Per\rp\in\mathcal{ATEH}$. Say $\lp G,\Per\rp\in \mathcal{TEH}$ if $\lp G,\Per\rp$ is toral relatively hyperbolic, i.e. $G$ is torsion-free and each $P\in\Per$ is abelian.
\end{definition}

Note that if $G$ is a toral relatively hyperbolic group, any $1$--ended virtually abelian group must be peripheral, and therefore is abelian. Since a torsion-free virtually cyclic group is cyclic, the hierarchy given for $G\in\mathcal{TEH}$ is in fact an abelian hierarchy.

\section{The BHH--JSJ hierarchy}\label{SectionBHHJSJ}

In this section, we assume $\lp G,\Per\rp$ is a relatively hyperbolic group where each $P\in \Per$ is virtually abelian and $1$--ended. The goal of this section is Theorem \ref{ATEHClosed} which is a more precise version of Theorem \ref{QIClosed} from the introduction. First, the BHH--JSJ hierarchy is defined. It is then shown that the BHH--JSJ hierarchy of a group in $\mathcal{ATEH}$ is finite in Theorem \ref{MakeGood}. Next we apply Lemma \ref{BulkLemma} \ref{QI1endedVertexGroups} and \ref{QIRigidVertexGroups} to complete the result.

If the boundary $\partial\lp G,\Per\rp$ is not connected, $G$ splits relative to $\Per$ over a finite subgroup. Using Theorem \ref{BowditchConnectedness} and Theorem \ref{BowditchAccess}, $G$ splits over finite groups relative to $\Per$, so that each vertex group is hyperbolic relative to the elements of $\Per$ that it contains. Additionally, each vertex group either finite or relatively $1$--ended. Now, using Theorem \ref{ConnectedGivesLocallyConnected} the boundary of each vertex group of this splitting is connected and locally connected. %Additionally, since $\partial\lp G,\Per\rp$ is locally connected, then by (Bowditch result in peripheral splittings paper) we have that the boundary of each vertex group is locally connected. 

 By Theorem \ref{HHJSJ}, the boundary of each relatively $1$--ended vertex group from the Bowditch splitting gives a canonical JSJ splitting over elementary subgroups relative to the peripheral structure. Vertex groups in this JSJ are  either peripheral, non-parabolic $2$--ended, quadratically hanging with finite fiber, or rigid. 
 
 \begin{definition}[The BHH--JSJ Hierarchy]
 Let $\lp G,\Per\rp$ be a relatively hyperbolic group. We define the BHH--JSJ of $\lp G,\Per\rp$ as follows. At any point, if $G$ is virtually free, virtually abelian, or virtually Fuchsian, we give $G$ a trivial splitting and it is terminal. Otherwise, if $G$ has more than one end relative to $\Per$, take the maximal splitting over finite groups relative to $\Per$ of Bowditch. For each vertex group, assign it the peripheral structure of the elements of $\Per$ that it contains. If $G$ is $1$--ended relative to $\Per$, take the canonical JSJ tree given by Haulmark and Hruska. For every rigid vertex group $H$, give $H$ the peripheral structure $\mathcal O$ where for each $O\in \mathcal O$, $O=gPg^{-1}\cap H$ for some $g\in G$ and $P\in\Per$. Remove any finite or $2$--ended elements from $\mathcal O$.
 \end{definition}
 
 The following observations will be useful throughout the remainder of the paper. If $H$ is a rigid vertex group in the JSJ splitting of Theorem \ref{HHJSJ}, $\mathcal O$ is still a peripheral structure on $H$ by Theorem \ref{ThrowOut2ended}, but in particular it contains only $1$--ended subgroups, allowing us to apply Lemma \ref{BulkLemma}. The BHH--JSJ hierarchy is elementary, as all of the splittings are over $2$--ended or peripheral subgroups. Additionally, the BHH--JSJ is relative to $\Per$. If the BHH--JSJ hierarchy is finite, it can only terminate in groups that are virtually free, virtually abelian (peripheral), virtually Fuchsian, or $1$--ended with no cut points or cut pairs in the boundary.
 
The following result allows one to consider the BHH--JSJ hierarchy of any group $G\in\mathcal{ATEH}$ as it satisfies the conditions of the required hierarchy in Definition \ref{ATEHDef}. In particular, given any satisfactory hierarchy, it is shown that the BHH--JSJ hierarchy is finite. We note that a result of Louder and Touikan \cite{LTStrongAccess}*{Corollary 2.7} could alternatively be used to show that the BHH-JSJ hierarchy is finite once it is known that groups in $\mathcal{ATEH}$ virtually do not contain 2-torsion.

\begin{theorem}\label{MakeGood}
Suppose $\lp G,\Per\rp\in\mathcal{ATEH}$. Then the BHH--JSJ hierarchy of $G$ is finite and terminates in virtually free, virtually abelian, or virtually Fuchsian groups.
\end{theorem}
\begin{proof}
Let $\mathcal{H}$ be a hierarchy of $G$ satisfying the assumptions of Definition \ref{ATEHDef}, given since it is assumed $G\in\mathcal{ATEH}$. First, note that by collapsing edges in the splittings of $\mathcal H$ and at most doubling the length of $\mathcal H$, $\mathcal H$ can be made to have finite edge groups at even depths, and infinite elementary edge groups at odd depths. The proof proceeds by induction on the depth of $\mathcal H$. If $\mathcal H$ has depth 0, then $G$ is virtually free, virtually abelian, or virtually Fuchsian, so the result follows trivially. Now, assume that the result holds for hierarchies with depth less than the depth of $\mathcal H$. It will be shown that every non-terminal vertex group of the BHH--JSJ hierarchy of $G$ inherits a strictly shorter hierarchy from $\mathcal H$. First, suppose that $\lp H,\Per'\rp$ is a relatively $1$--ended vertex group at depth 1 of the BHH--JSJ hierarchy. Then, since this vertex group cannot split over finite groups, it acts elliptically on the tree at depth 0 of $\mathcal H$. Therefore $\lp H, \Per'\rp$ is a subgroup of some vertex group of this tree and inherits a strictly shorter hierarchy from $\mathcal H$, as it is contained in a vertex group of depth at least 1. 

Now, suppose that $H$ is a rigid vertex group of the canonical JSJ splitting at depth 2 of the BHH--JSJ hierarchy. Then $H$ is a  subgroup of a vertex group $H'$ at depth 1 of the BHH--JSJ hierarchy. Since $H'$ is $1$--ended, from above $H'$ is the subgroup of a vertex group $K$ at depth at least 1 of $\mathcal H$. For convenience, we can assume that $K$ is chosen to be the vertex group of $\mathcal H$ of maximal depth containing $H'$. 

If $K$ splits non-trivially over elementary groups in $\mathcal H$, $H$ must act elliptically on the splitting of $K$, since $H$ is a rigid vertex group of the JSJ of $H'$ over elementary subgroups relative to the peripheral structure. This gives a strictly shorter hierarchy for $H$ inherited from a vertex group fixed by $H$ in the splitting of $K$. If the splitting of $K$ over elementary subgroups is trivial, $K$ may not split over finite groups relative to its peripheral structure, since $H'$ fixes a vertex in this tree (since $H'$ is $1$--ended). Otherwise, $H'$ is contained in a deeper vertex group of $\mathcal H$, contradicting our choice of $K$. In this case, since $K$ does not split, it is terminal, and, since $K$ is $1$--ended, $H$ is either virtually abelian, or virtually Fuchsian, and terminal, so no such $H'$ can exist.
\end{proof}

\begin{theorem}\label{ATEHClosed} Let $G\in \mathcal{ATEH}$ and $\lp G,\Per\rp$ be relatively hyperbolic where each $P\in \Per$ is virtually abelian and $1$--ended. If $G'$ is a group quasi-isometric to $G$, then $G'$ admits a peripheral structure $\Per'$ and a hierarchy $\mathcal H'$ so that $\mathcal H'$ is over elementary subgroups of $G'$ relative to $\Per'$ and terminates in virtually free, virtually abelian, or virtually Fuchsian subgroups. In particular, $G'\in\mathcal{ATEH}$. 
\end{theorem}
\begin{proof}
Lemma \ref{RelVAb} gives us the peripheral structure $\Per'$. Now using Theorem \ref{MakeGood}, replace $\mathcal H$ with the BHH--JSJ hierarchy so that $\mathcal H$ is still finite, and terminates in virtually free, virtually abelian, or virtually Fuchsian groups, but each splitting is a JSJ splitting. Let $\mathcal H'$ be the BHH--JSJ hierarchy for $\lp G',\Per'\rp$. For simplicity, if $G$ (and therefore $G'$) are relatively $1$--ended, we consider $\mathcal H$ and $\mathcal{H}'$ to have a trivial splitting at depth 0, and then a splitting over infinite elementary groups at depth 1, so that groups at even depth split over finite subgroups, and groups at odd depths split over infinite elementary subgroups.

Lemma \ref{BulkLemma}\ref{QI1endedVertexGroups} shows that the quasi-isometry types of relatively $1$--ended vertex groups of the tree at depth 0 of $\mathcal H$ and $\mathcal H'$ are the same. 

Let $H$ be a vertex group of the splitting of $G$ with the induced peripheral structure $\mathcal O$ so that $H$ is relatively $1$--ended. By the previous paragraph, there is a quasi-isometric vertex group $H'$ at depth 1 of $\mathcal H'$. By Lemma \ref{BulkLemma}\ref{QIRigidVertexGroups} all of the rigid groups of the JSJ of $H'$ are quasi-isometric to a rigid group in the JSJ of $H$. This holds for every vertex group that labels a non-terminal vertex of depth 1, so the set of quasi-isometry types of rigid vertex groups of the JSJ splittings at depth 1 of $\mathcal H$ is the same as the set of quasi-isometry types of rigid vertex groups of the JSJ splittings at depth 1 of $\mathcal H'$. Since all other vertex groups in the JSJ trees are terminal, the result follows from induction. 
\end{proof}

\section{Separability}\label{SectionSeparability}

The goal of this section is to prove Theorem \ref{ATEHisLerf};  a group $G\in \mathcal{ATEH}$ is LERF. We prove this in two parts. First we show that $G$ is locally relatively quasi-convex. Then we show that $G$ virtually embeds in a virtually special group. An application of a separability result of Sageev and Wise and standard arguments about quasi-convexity and separability shows the result. 

The first lemma follows by adapting the work of Dahmani in \cite{Dahmani}. Similarly, the work of Bigdely and Wise in \cite{BigdelyWiseLQC} could have been adapted to avoid the restriction to small terminal vertex groups. As the boundary is considered throughout this paper, an adaptation of the work of Dahmani seemed more appropriate. 

\begin{theorem}\label{lrqc}
Suppose that $G\in\mathcal{ATEH}$. Then $G$ is locally relatively quasi-convex.
\end{theorem}
\begin{proof}
We follow the proof of \cite{Dahmani}*{Proposition 4.6} using elements of the proof of \cite{Dahmani}*{Theorem 0.1}. We use induction on the number of edges in a splitting and depth of the hierarchy. The result holds for free, abelian, and surface groups. Since quasi-convexity is invariant under commensurability, the result holds for virtually free, virtually abelian, and virtually Fuchsian groups.

Now, suppose that $G=A*_C B$ with $A,B$ locally relatively quasi-convex and $C$ elementary. Let $T$ be the Bass-Serre tree for this splitting. If $C$ is finite, or non-parabolic virtually cyclic, then $C$ is almost malnormal in $A$ and $B$, so $T$ is acylindrical. Furthermore, $C$ is either finite, or non-peripheral $2$--ended, so the intersection $C\cap gPg^{-1}$ is finite for all $g\in G$ and $P\in\Per$. Therefore $C$ is fully quasi-convex in $A$ and $B$. Let $H$ be a finitely generated subgroup of $G$. Since every subgroup of $C$ is finitely generated, $H$ acts on $T$ with finitely generated edge stabilizers, so the vertex stabilizers of the $H$ action on $T$ must also be finitely generated. Therefore, the  intersection of $H$ with each conjugate of $A$ or $B$ is finitely generated, so by the inductive hypothesis, $H\cap hAh^{-1}$ is relatively quasi-convex in $A$ and $H\cap hBh^{-1}$ is relatively quasi-convex in $B$ for all $h\in H$. Letting $T'$ be the minimal $H$ invariant subtree of $T$, then $H\backslash T'$ is a finite splitting of $H$ with finitely generated edge and vertex groups. Since the vertex groups of $H\backslash T'$ are relatively quasi-convex in the conjugates of $A$ and $B$, the boundary of each vertex group embeds in a translate of $\partial A$ or $\partial B$ in $\partial G$. Since $T'$ embeds in $T$, Dahmani's construction of the boundary $\partial H$ from the splitting $H\backslash T'$ shows that $\partial H$ embeds in $\partial G$ as the limit set of $H$. Therefore, $H$ acts as a geometrically finite group on its limit set, and is therefore relatively quasi-convex in $G$. 

In the case that $C$ is peripheral, let $P$ be the maximal peripheral group containing $C$. Let $C_A=A\cap P$ and $C_B=B\cap P$. Since we assume splittings to be relative to $\Per$, either $C_A=P$ or $C_B=P$. Without loss of generality, we can assume $C_A=P$ and write the splitting as $G= A*_P\lp P*_{C_B} B\rp$. Note that the boundary of $B'=P*_{C_B}B$ is built by Dahmani in Case (2) of \cite{Dahmani}*{Theorem 0.1}, and from that the boundary of $A*_P B'$ is built in Case (1) since the subgroup $P$ is maximal peripheral in $A$ and $B'$. If $H$ is a finitely generated subgroup, since $B$ is locally relatively quasi-convex, $H\cap B'$ is relatively quasi-convex using the construction of the boundary of $B'$, and so $H$ is relatively quasi-convex in $G$ using the construction of the boundary of $A*_P B'=G$. 

Finally, in the case that $G=A*_{tC_1t^{-1}=C_2}$, if $C_1$ and $C_2$ are finite or non-parabolic $2$--ended, $C$ is again almost malnormal and fully quasi-convex in $A$, and the proof follows as before using the construction of the boundary in Case (1) of \cite{Dahmani}*{Theorem 0.1}. Suppose $C_1$ is infinite parabolic. Let $P_i$ be the maximal parabolic subgroup of $G$ containing $C_i$, and note that since $C_2\subset P_1\cap tP_2t^{-1}$, by almost malnormality of the peripheral structure on $G$, $P_2=tP_1t^{-1}$. Now, the peripheral structure on $A$ is given by $gPg^{-1}\cap A$ for $g\in G$ and $P\in \Per$, so we note that $C_1$ and $C_2$ are infinite peripheral in $A$, and therefore relatively quasi-convex. 

Let $C_1\subseteq D_1$ and $C_2\subseteq D_2$ where each $D_i$ is maximal peripheral in $A$. Suppose that $d_1\in D_1\setminus C_1$ and $d_2\in D_2\setminus C_2$. Then $td_1t^{-1}\in P_2$ and stabilizes the vertex $tA$ in the Bass-Serre tree corresponding to the splitting. Additionally, since $td_1t^{-1}\not\in C_2$ it does not stabilize the edge from $A$ to $tA$. Similarly, $d_2\in A\setminus C_2$, so it stabilizes the vertex $A$, but not the edge from $A$ to $tA$. Therefore, $P_2$ cannot stabilize a vertex, but we assumed the splitting was relative to $\Per$, so this is a contradiction. Therefore either $C_1$ or $C_2$ is maximal peripheral in $A$. Without loss of generality, suppose that $C_1$ is maximal peripheral in $A$.

Consider the splitting $G=\lp A*_{C_1}D_2\rp*_{D_2}$ where the attaching maps of the HNN extension are given by the identity on $D_2$ and the inclusion into $A$. Letting $A'=A*_{C_1}D_2$, we see that $A'$ is relatively hyperbolic and locally quasi-convex as before by using Case (2) of \cite{Dahmani}*{Theorem 0.1}. Then we can apply Case (1) of \cite{Dahmani}*{Theorem 0.1} again to see that $G$ is locally quasi-convex.
%Now suppose that $C_1$ is not a finite index subgroup of $P_1\cap A$. Let $P_1'$ be a finite index abelian subgroup of $P_1$,  and let $c_1\in\lp P_1'\cap A\rp\setminus C_1$. Suppose also that $C_2$ is not finite index in $P_2\cap A$. Then, letting $P_2'=tP_1't^{-1}$, we note that $P_2'$ is abelian and finite index in $P_2$, and we can find a $c_2\in \lp P_2'\cap A\rp\setminus C_2$. Then we reach a contradiction as $tc_1t^{-1}$ and $c_2$ should commute as they are both in $P_2'$, however this cannot be recovered by the relations in the HNN extension. Therefore either $C_1$ is finite index in $P_1\cap A$ or $C_2$ is finite index in $P_2\cap A$. Without loss of generality, suppose that $C_1$ is finite index in $P_1\cap A$. Let $\hat{P_2}=P_2\cap A$ and consider the splitting $G=\lp A*_{C_2} P_2'\rp*_{P_1'}$, where the attaching map of the HNN extension is given by the identity on $P_1'$ and the image in $A$ identifying $C_1$ with $C_2$. As $C_1$ is finite index in $P_1\cap A$, it is finite index in a maximal parabolic subgroup of $A$ and therefore fully relatively quasiconvex in $A$.    
\end{proof}

Next, we show that each $G\in\mathcal{ATEH}$ virtually embeds in a special group. As special groups have many separability properties, Theorem \ref{ATEHisLerf} follows shortly after. The next lemma allows us to apply results of Wise to groups in $\mathcal{ATEH}$, but first, we need the notion of a strongly sparse special cube complex. The following definitions appear in \cite{WiseQCH}.

\begin{definition}[Strong Quasiflats \cite{WiseQCH}*{Definition 7.18}]
A \emph{strong quasiflat} is a locally finite CAT$\lp 0\rp$ cube complex $\tilde{F}$, with an action on $\tilde{F}$ by a finitely generated virtually abelian group $P$ so that there are finitely many $P$ orbits of hyperplanes, there are finitely many $P$ orbits of pairs of osculating hyperplanes, and there are finitely many $P$ orbits of pairs of intersecting hyperplanes.
\end{definition}

We note as an example that the standard cubulation of $\mathbb{R}^n$ being acted on by $\mathbb{Z}^n$ in the standard way is a strong quasiflat.

\begin{definition}[Strongly Cosparsely \cite{WiseQCH}*{Definition 7.21}]
If $G$ is hyperbolic relative to finitely generated virtually abelian groups $\{P_i\}$, $G$ is said to act \emph{strongly cosparsely} on a CAT$\lp 0\rp$ cube complex $\tilde{X}$ if there is a compact subcomplex $K$, and strong quasiflats $\tilde{F}_i$ so that $P_i$ is the stabilizer of each $\tilde{F}_i$, and \begin{itemize}
    \item $\tilde{X}=GK\cup \bigcup_i G\tilde{F}_i.$
    \item For each $i$, $\lp \tilde{F}_i\cap GK\rp\subseteq P_iK_i$ for some compact $K_i$.
    \item For $i,j$, and $g\in G$, either $\lp \tilde{F}_i\cap g\tilde{F}_j\rp\subseteq GK$ or else $i=j$ and $\tilde{F}_i=g\tilde{F}_j$.
\end{itemize}
If $G$ acts freely and strongly cosparsely on $\tilde{X}$, we say that $X=G\backslash \tilde{X}$ is \emph{strongly sparse}.
\end{definition}

Special cube complexes were introduced in \cite{HaglundWise}. We point the reader to this paper for the definition of special cube complexes. The following is well known, with multiple proofs such as \cite{WiseQCH}*{Example 6.1}.

\begin{lemma}
Suppose that $\Gamma$ is a finite graph. Then the Salvetti complex of the right-angled Artin group associated with $\Gamma$ is compact and special. \end{lemma}

The following lemma shows that groups in $\mathcal{ATEH}$ have hierarchies terminating in strongly sparse special groups, which allows us to apply \cite{WiseQCH}*{Remark 18.17}.

\begin{lemma}\label{use1817}
Suppose $\lp G,\Per\rp\in\mathcal{ATEH}$. Then $G$ is hyperbolic relative to virtually abelian subgroups, and has a virtually abelian hierarchy terminating in virtually strongly sparse special groups.
\end{lemma}
\begin{proof}
We know that $G$ is hyperbolic relative to virtually abelian subgroups by the definition of $\mathcal{ATEH}$. The BHH--JSJ hierarchy of $\lp G,\Per\rp$ is finite and over elementary subgroups. Since the peripheral groups are virtually abelian, the BHH--JSJ hierarchy is over virtually abelian groups, so all that needs to be shown is that the terminal vertex groups are strongly sparse special. 

Finitely generated free groups and finitely generated abelian groups are special as they are right-angled Artin groups of a finite graph. Since the peripheral structure on virtually free groups is empty, they are virtually strongly sparse special. The conditions for being strongly sparse are satisfied also when the peripheral subgroup is the entire group, so virtually abelian groups are also virtually strongly sparse special. This leaves only virtually Fuchsian groups.

Let $\mathbf{\mathcal{QVH}}$ be the collection of all virtually special hyperbolic groups. This is equivalent to the definition in \cite{WiseQCH}*{Definition 11.5} by \cite{WiseQCH}*{Theorem 13.3}. Recall that a virtually Fuchsian group $H$ arising as a vertex of the third type in Theorem \ref{HHJSJ} fits into a short exact sequence $$1\rightarrow F\rightarrow H \rightarrow \pi_1\lp \Sigma\rp\rightarrow 1$$ where $\Sigma$ is a hyperbolic orbifold. Hyperbolic surface groups are virtually special as they are amalgamated free products of free groups over cyclic groups. Therefore, $\pi_1\lp \Sigma\rp$ is virtually special, as it has a finite index subgroup which is a hyperbolic surface group.  By \cite{AgolVH}*{Lemma 2.10}, if $\phi:H\rightarrow G$ is a surjective homomorphism with finite fiber and $G\in\mathbf{\mathcal{QVH}}$ then $H\in\mathbf{\mathcal{QVH}}$. Therefore, if $H$ is virtually Fuchsian, $H$ is virtually special. Again, the peripheral structure on the virtually Fuchsian vertex group is trivial, so these terminal vertex groups are virtually strongly sparse special.
\end{proof}

Immediately from \cite{WiseQCH}*{Remark 18.17}, we have:
\begin{corollary}\label{VSpecialNonCompact}
Suppose $\lp G,\Per\rp\in\mathcal{ATEH}$, then $G$ has a finite index subgroup that is the fundamental group of a strongly sparse special cube complex.
\end{corollary}

%reference wise for definition

%%write as corollary to 18.17 that my groups are

\begin{theorem}\label{VirtualEmbedInVirtualSpecial}
If $G\in\mathcal{ATEH}$, then $G$ has a finite index subgroup that embeds in a virtually compact special group.
\end{theorem}
\begin{proof}
By Corollary \ref{VSpecialNonCompact}, $G$ has a finite index subgroup $G'$ which is the fundamental group of a strongly sparse special cube complex $X$. Note that the induced peripheral structure on $G'$ shows that $G'$ is still hyperbolic relative to virtually abelian subgroups. Therefore, $G'$ embeds in a group $\hat{G}$ which is the fundamental group of a compact non-positively curved cube complex $\hat{X}$ by \cite{WiseQCH}*{Theorem 7.54}. We note that $\hat{G}$ is still hyperbolic relative to virtually abelian groups by the tree structure of the splitting given in \cite{WiseQCH}*{Theorem 7.54} and Case (2) of Dahmani's combination theorem \cite{Dahmani}*{Theorem 0.1}. It is therefore sufficient to show that $\hat{G}$ is virtually special. This follows by \cite{oregonvspecial}*{Corollary 1.3} since $\hat{G}$ is hyperbolic relative to virtually abelian subgroups and acts properly and cocompactly on the universal cover of $\hat{X}$, which is CAT$(0)$.
\end{proof}

\begin{theorem}\label{lerfness}
Suppose that $G\in\mathcal{ATEH}$. Then $G$ is LERF.
\end{theorem}
\begin{proof}
First, by Lemma \ref{lrqc}, $G$ is locally relatively quasi-convex, so it is sufficient to show that for any relatively quasi-convex subgroup $H$ of $G$ that $H$ is separable.

Let $G'$ and $\hat{G}$ be as in the proof of Theorem \ref{VirtualEmbedInVirtualSpecial}. Let $\tilde{G}$ be a finite index special subgroup of $\hat{G}$. Since $G'$ is relatively quasi-convex in $G$, $H'=G'\cap H$ is relatively quasi-convex in $G'$. Note that $G'$ is relatively quasi-convex in $\hat{G}$ by Dahmani's combination theorem, so $H'$ is also relatively quasi-convex in $\hat{G}$. Let $\tilde{H}=H'\cap\tilde{G}$. $\tilde{G}$ is hyperbolic relative to virtually abelian subgroups. Each subgroup of a virtually abelian group $P$ is separable in $P$, so we can apply \cite{SageevWiseCores}*{Corollary 6.4} to see that $\tilde{H}$ is separable in $\tilde{G}$.  

Now, since $\tilde{H}$ is closed in the profinite topology of $\tilde{G}$, we know that it is closed in the profinite topology of $\hat{G}$. Since $\tilde{G}$ is finite index in $\hat{G}$, $\hat H$ is the finite union of translates of $\tilde{H}$, so that $\hat{H}$ is also closed in the profinite topology of $\hat{G}$ as it is the finite union of closed sets. Writing $\hat{H}=\bigcap K_i$ where each $K_i$ is finite index in $\hat{G}$ shows each $K_i\cap G'$ is finite index in $G'$, which shows that $H'=\bigcap\lp K_i\cap G'\rp$, so that $H'$ is separable in $G'$. 
\end{proof}

The next two results are meant to strengthen the comparison between groups in $\mathcal{ATEH}$ and virtual limit groups.

\begin{lemma}\label{GisVTF}
Suppose that $G\in \mathcal{ATEH}$. Then $G$ is virtually torsion-free.
\end{lemma}
\begin{proof}
Since $G\in\mathcal{ATEH}$, by \cite{WiseQCH}*{Remark 18.17}, we know that $G$ has a finite index subgroup $G'$ that is the fundamental group of a sparse special cube complex. Therefore, $G'$ acts properly discontinuously on the universal cover of this cube complex, which is a CAT$(0)$ space and is aspherical, so $G'$ has finite cohomological dimension, and is torsion-free.
\end{proof}

\begin{lemma}\label{GisVHRelAb}
Suppose that $G\in \mathcal{ATEH}$. Then $G$ has a finite index subgroup which intersects conjugates of elements of $\Per$ in abelian subgroups, and so is hyperbolic relative to abelian subgroups.
\end{lemma}
\begin{proof}
Using Lemma \ref{GisVTF}, we may assume that $G$ is torsion-free, and in particular this gives that $\Per$ is a malnormal collection of virtually abelian groups. Suppose $\Per=\{P_1,...,P_n\}$. For each $i\in \{1,...,n\}$ we construct a $G_i$ which is finite index in $G$ and so that $G_i\cap P_i$ is abelian. Note that each $P_i$ is assumed to be $1$--ended, and that every non-cyclic abelian subgroup of $G$ must be contained in a peripheral subgroup.

Take $P_i'$ to be a maximal finite index abelian subgroup of $P_i$ for each $i$. First, we claim that each $P_i'$ is separable in $G$. Indeed, suppose that $a\in G\setminus P_i'$, so that $\lb a,p\rb\neq 1$ for some $p\in P_i'$. Then, as $G$ is residually finite, there is a homomorphism $h:G\rightarrow Q$ where $Q$ is finite, and $h\lp \lb p,a\rb\rp\neq 1$. Since $h\lp A\rp$ is abelian, we know that $h\lp a\rp\not\in h\lp P_i'\rp$. Therefore, $P_i'$ is separable in $G$. 

Now, let $\{1,k_1^i...,k^i_{m_i}\}$ be a set of coset representatives of $P_i'$ in $P_i$. Using the fact that $G$ is LERF, for each $j\in\{1,...,m_i\}$ take $G_i^j$ to be a finite index subgroup of $G$ containing $P_i'$ but not $k_j^i$. Then $G_i=\bigcap_{j=1}^{m_i} G_i^j$ is finite index in $G$ and contains $P_i'$. Moreover, $G_i\cap P_i=P_i'$. 

Lastly, taking $G'=\bigcap_{i=1}^{n} G_i$ we get a finite index subgroup of $G$ so that $P_i\cap G'$ is abelian for each $i$. Taking $G''$ to be a finite index normal subgroup of $G$ contained in $G'$, $G''\cap P_i\subset P_i'$ for all $i$. Furthermore, since $G''$ is normal in $G$, for each $g\in G$, $G''\cap gP_ig^{-1}=g\lp G''\cap P_i\rp g^{-1}\subseteq g\lp G'\cap P_i'\rp g^{-1}$, which is abelian, so that $G''$ is hyperbolic relative to the collection of infinite conjugates of $\{G''\cap P_i'\}_{i=1}^n$. 
\end{proof}

For any $G\in\mathcal{ATEH}$, by taking the intersection of the finite index torsion-free subgroup given by Lemma \ref{GisVTF} and the finite index subgroup in Lemma \ref{GisVHRelAb}, we get a finite index toral relatively hyperbolic subgroup of $G$. Since this finite index subgroup has the same Bowditch boundary as $G$, the BHH--JSJ hierarchy is finite and terminates in free, abelian, or hyperbolic orbifold groups. The following is immediate.

\ThmATEHisVTEH*

\section{Cyclic Hierarchies}\label{SectionCyclicHierarchies}

The final section of this paper is spent showing that groups in $\mathcal{ATEH}$ also have finite virtually cyclic hierarchies. First, we show that non-terminal vertex groups in the given hierarchy split over virtually cyclic groups. Examining the boundary allows us to show that collapsing all $1$--ended edge groups in the JSJ given by Theorem \ref{HHJSJ} is a JSJ in the class of splittings over $2$--ended subgroups relative to $\Per$. This allows us to apply a result of Louder and Touikan to see that the BHH--JSJ hierarchy is still finite if the $1$--ended edge groups are all collapsed.

\begin{lemma}\label{SplitsCyclically}
Let $\lp G, \Per\rp\in\mathcal{ATEH}$ and suppose that $G$ is not virtually abelian, virtually free, or virtually Fuchsian. Then $G$ splits over a finite or virtually cyclic subgroup.
\end{lemma}
\begin{proof}
First, suppose that the BHH--JSJ hierarchy of $G$ has depth one. If $G$ is not relatively $1$--ended, the maximal Bowditch splitting is relative to $\Per$ and over finite subgroups, so the claim holds. If $G$ is relatively $1$--ended, consider the Haulmark-Hruska JSJ of $G$.  Since the JSJ tree is bipartite, no two peripheral groups are adjacent. Virtually Fuchsian vertex groups are not virtually non-cyclic abelian. A rigid vertex group has boundary that is not a point, but each rigid vertex group is hyperbolic relative to the peripheral groups it contains or intersects. In particular, since rigid vertex groups and quadratically hanging vertex groups are not virtually abelian, no two virtually abelian groups are adjacent in the JSJ tree of $G$. Since virtually free and virtually Fuchsian groups do not contain non-cyclic abelian subgroups, all edges in this splitting must be $2$--ended, and again the claim holds.

Now suppose that $G$ is not relatively $1$--ended, and has BHH--JSJ hierarchy of depth 2. The splittings of Bowditch is again over virtually cyclic subgroups, and so the claim holds.

Suppose $G$ is relatively $1$--ended, and has BHH--JSJ of depth 2. Suppose that there are no edge stabilizers in the JSJ of $G$ that are $2$--ended. This implies that each edge group is $1$--ended. Suppose that $G_e$ is the edge stabilizer of an edge incident on a vertex $v$ stabilized by $G_v$, and that $G_v$ is non-terminal. This must happen as we assumed the depth of the hierarchy was 2. Since $G_v$ splits over finite groups relative to its induced peripheral structure, $G_e$ is conjugate into one of the vertices in this splitting of $G_v$. However, this is true of each edge incident on $v$, so that the JSJ can be refined by replacing $v$ with the splitting of $G_v$. This indicates that $G$ was not relatively $1$--ended, as it splits relative to $\Per$ over finite groups, a contradiction.

Suppose that the BHH--JSJ hierarchy of $G$ has depth $n$.  By the arguments above a vertex group at depth $n-1$ or $n-2$ splits over a virtually cyclic subgroup. Suppose there  is a vertex group $H$ at depth $k$ that does not split over a virtually cyclic subgroup, chosen so that $k$ is maximal. It will be shown that $H$ is terminal. If the induced hierarchy of $H$ has depth 1 or 2, then as before, $H$ splits over a virtually cyclic subgroup. If $H$ has more than one end, then $H$ splits over a virtually cyclic subgroup. If $H$ is virtually $1$--ended, by assumption the Haulmark-Hruska JSJ of $H$ has only $1$--ended peripheral edge groups. These edge groups act elliptically on the Bowditch splitting of each vertex group $H_v$, as they are $1$--ended and do not split over finite groups, and so each non-terminal vertex group must have trivial Bowditch splitting, else again the JSJ could be refined and $H$ would not be relatively $1$--ended. By the assumption that $k$ was maximal, we must have that each rigid vertex group $H_v$ in the JSJ of $H$ has a cyclic edge in its JSJ splitting. However, since each edge incident on $v$ has edge group that is peripheral in $H_v$, it acts elliptically on the JSJ of $H_v$. Therefore, the JSJ can be refined by replacing $v$ by the JSJ of $H_v$, and so $H$ does split over a $2$--ended subgroup since $H_v$ does. This is a contradiction, so $H$ must be terminal. 
\end{proof}

In particular, if the BHH--JSJ hierarchy of $\lp G,\Per\rp$ is not trivial, then $G$ splits over a virtually cyclic subgroup.

Define the \emph{collapsed JSJ} $\mathcal T'$ for a relatively $1$--ended $G\in\mathcal{ATEH}$ as follows. Let $\mathcal T$ be the canonical JSJ of Haulmark and Hruska. We obtain $\mathcal T'$ from $\mathcal T$ by collapsing all edges that are $1$--ended. By the construction of the JSJ, all edges collapsed are between peripheral vertex groups and rigid pieces that are not quadratically hanging. Additionally, by \cite{HaulmarkLocalCutPoints}*{Theorem 1.3}, we know that the number of ends of an edge group, say between a peripheral group $P$ and a rigid group $H$, is the same as the number of ends of $C\setminus p$ where $p$ is the local cut point corresponding to the group $P$, and $C$ is the limit set of $H$. In particular, it is therefore possible to identify which edges of the JSJ tree are being collapsed from only the Bowditch boundary of $\lp G,\Per\rp$.

Recall the following from Section \ref{SectionSplittings}. A $G$-tree $T$ is an $\lp \mathcal E,\mathcal S\rp$-tree for classes of groups $\mathcal E$ and $\mathcal S$ if all edge groups are in $\mathcal E$ and all groups in $\mathcal S$ act elliptically on $T$.
A $\lp \mathcal E,\mathcal S\rp$-tree $T$ is universally elliptic if each edge stabilizer of $T$ acts elliptically on any other $\lp \mathcal E,\mathcal S\rp$-tree. An $\lp \mathcal E,\mathcal S\rp$-tree $T$ dominates a tree $T'$ if there is a $G$-equivariant map $T\rightarrow T'$. Let $\mathcal E$ be the collection of virtually cyclic subgroups of $G$. A \emph{virtually cyclic JSJ} of $\lp G,\Per\rp$ is a $\lp \mathcal E,\Per\rp$-tree $T$ that is universally elliptic, and dominates any other universally elliptic $\lp \mathcal E,\Per\rp$-tree.

\begin{lemma}\label{IsJSJ}
Suppose that $\lp G,\Per\rp\in \mathcal{ATEH}$ and that $G$ is relatively $1$--ended. Then the collapsed JSJ for $G$ is a virtually cyclic JSJ for $G$.
\end{lemma}
\begin{proof}
Let $\mathcal T'$ be the tree corresponding to the collapsed JSJ of $G$ and let $\mathcal T$ be the tree corresponding to the JSJ of $G$ from Theorem \ref{HHJSJ}. First, we note that non-parabolic $2$--ended vertex groups and quadratically hanging groups of $\mathcal T$ have only $2$--ended incident edge groups, and therefore are still vertex stabilizers of the action on $\mathcal T'$. We note that edge groups in $\mathcal T'$ are edge groups in $\mathcal T$, so $\mathcal T'$ is universally elliptic since $\mathcal T$ is.
Next we must show that $\mathcal T'$ dominates any other universally elliptic $\lp \mathcal E,\Per\rp$-tree $\mathcal S$ where $\mathcal S$ has $2$--ended edge groups and is relative to $\Per$. By definition, any peripheral vertex groups in $\mathcal T'$ are elliptic in $\mathcal S$. The fact that non-peripheral $2$--ended vertex groups and quadratically hanging vertex groups act elliptically on $\mathcal S$ is shown in the same way as in the proof of \cite{HHCanonSplittings}*{Theorem 9.1}. The proof is recalled here. If $G_v$ is a $2$--ended, non-peripheral vertex group of $\mathcal T'$, then it has limit set $\{\eta,\chi\}\in\partial\lp G,\Per\rp$. If this limit set coincides with the limit group of an edge group $G_e$ of $\mathcal S$, then $G_v$ and $G_e$ are co-elementary, and $G_v$ has a finite index subgroup acting elliptically on $\mathcal S$, so $G_v$ acts elliptically on $\mathcal S$ as well. If $\{\eta,\chi\}$ is disjoint from the limit set of each edge group in $\mathcal S$, then $\{\eta,\chi\}$ is contained in the limit set of a vertex group of $\mathcal S$ by \cite{HHCanonSplittings}*{Proposition 6.9}, and therefore fixes that vertex in $\mathcal S$.

Now suppose that $G_v$ is quadratically hanging. By \cite{GLJSJs}*{Lemma 2.15} $\mathcal S$ can be refined to a JSJ $\hat{\mathcal S}$ and by \cite{GLJSJs}*{Theorem 5.27}, $G_v$ acts elliptically on $\hat{\mathcal S}$, and therefore also on $\mathcal S$.

If $G_v$ is rigid, and comes from a rigid vertex in $\mathcal T$, $G_v$ would act elliptically on any splitting of $G$ over elementary subgroups, and in particular it acts elliptically on any splitting over $2$--ended subgroups. 

Lastly, for $v\in \mathcal T'$, $G_v$ may be an amalgamation over $1$--ended edge groups of peripheral groups and rigid vertices of $\mathcal T$. We claim that the limit set $\Lambda G_v$ cannot contain an exact cut pair of $\partial\lp G,\Per\rp$ or a global cut point $p$ of $\partial\lp G,\Per\rp$ where one of the components $C$ of $\partial\lp G,\Per\rp$ is such that $C\setminus p$ is $2$--ended. This follows as in either case there is a $2$--ended edge of $\mathcal T$ that would not be collapsed in $\mathcal T'$. Since $\Lambda G_v$ does not contain such topological features, it is contained in a half-space of each edge of $\mathcal S$, and by \cite{HHCanonSplittings}*{Lemma 6.9}, $\Lambda G_v$ is contained in the limit set of some vertex of $\mathcal S$. Therefore, $G_v$ stabilizes that vertex in $\mathcal S$.

Therefore $\mathcal T'$ dominates any universally elliptic splitting $\mathcal S$, so $\mathcal T'$ is a JSJ for $G$ over $2$--ended subgroups relative to $\Per$.
\end{proof}

\begin{theorem}
Suppose $(G,\Per)\in\mathcal{ATEH}$. Then $G$ has a finite hierarchy relative to $\Per$ over virtually cyclic subgroups terminating in virtually free, virtually abelian, or virtually Fuchsian groups.
\end{theorem}
\begin{proof}
We let $\mathcal H$ be the hierarchy for $\lp G,\Per\rp$ where at each level we take the splitting over finite groups relative to $\Per$ of Bowditch, and then the collapsed JSJ from Lemma \ref{IsJSJ}. Call this hierarchy the cyclic JSJ hierarchy. We assume all peripheral and virtually Fuchsian groups are terminal, and give non-terminal vertex groups the induced peripheral structure, without any $2$--ended groups included. This is a JSJ hierarchy for $G$ over virtually cyclic subgroups relative to $\Per$. Let $G_0\leq G$ be a finite index, torsion-free subgroup of $G$, which exists by Lemma \ref{GisVTF}. We note that since $\mathcal H$ can be determined entirely from the topology of the boundary $\partial\lp G,\Per\rp$, the cyclic JSJ hierarchy of $G_0$, $\mathcal H'$, is defined in the same way, so that $\mathcal H'$ is finite if and only if $\mathcal H$ is. Thus, we can take $G$ to be torsion-free. In this case, since $G$ has no 2-torsion, there cannot be any slender subgroups of $G$ which act dihedrally on a tree in $\mathcal H$, so $\mathcal H$ is a hyperbolic hierarchy. In particular, every slender subgroup of $G$ is either elliptic or fixes an end of any tree in $\mathcal H$. As all edge groups are virtually cyclic, and therefore Noetherian, $\mathcal H$ satisfies the ascending chain condition needed in \cite{LTStrongAccess}*{Theorem 2.5}. Since $G$ is assumed to be finitely presented, it is $\mathcal H$-almost finitely presented. Then by \cite{LTStrongAccess}*{Theorem 2.5}, there exists an $N$ and $C$ so that each vertex group $G_v$ of depth at least $N$ has a hierarchy $\mathcal K_v$ of height at most $C$ whose terminal vertex groups are either $\mathcal H$-elliptic or slender. 

Let $G_v$ be a group in $\mathcal H$ of depth at least $N$. We note that $G_v$ inherits a peripheral structure from $G$ and a hierarchy from $\mathcal H$. If $G_v$ is not terminal, the inherited hierarchy is a JSJ hierarchy. Using the same argument from Theorem \ref{MakeGood}, with the argument of Lemma \ref{IsJSJ} in place of the argument for rigid vertex groups, the existence of a finite virtually cyclic hierarchy $\mathcal K_v$ relative to $\Per$ implies that the inherited JSJ hierarchy of $G_v$ from $\mathcal H$ is finite. In particular, $\mathcal H$ is finite.

Now, it only remains to be shown that the terminal vertex groups of $\mathcal H$ are virtually free, virtually abelian, or virtually Fuchsian. Suppose that $G_v$ is a terminal vertex group of $\mathcal H$. Then $G_v\in\mathcal{ATEH}$ since $G\in\mathcal{ATEH}$. If $G_v$ is not virtually free, virtually abelian, or virtually Fuchsian, then $G_v$ splits over a virtually cyclic subgroup by Lemma \ref{SplitsCyclically}, so it is not terminal in $\mathcal H$.
\end{proof}

\bibliography{bibliography}

\end{document}